\documentclass{amsart}

\usepackage{amsfonts,amssymb}
 \usepackage{xspace,xcolor}
 \usepackage{graphicx,enumerate}
 
 \usepackage{color}
\usepackage[colorlinks,citecolor=blue]{hyperref}

\theoremstyle{plain}
\newtheorem{theorem}{Theorem}[section]

\newtheorem{lemma}[theorem]{Lemma}
\newtheorem{corollary}[theorem]{Corollary}

\newtheorem{proposition}[theorem]{Proposition}

\newtheorem{prop-defn}[theorem]{Proposition-Definition}

\newtheorem*{theorem:main}{Main Theorem} 
\theoremstyle{definition}

\newtheorem{question}[theorem]{Question}
\newtheorem*{remark*}{Remark}
\newtheorem*{remarks*}{Remarks}

\newcommand{\M}{\mathcal M}
\newcommand{\Mod}{\rm Mod}

\newcommand{\bY}{{\bf Y}}
\newcommand{\diam}{\rm diam}

\def\C{{\mathcal C}}

\title{Undistorted purely pseudo-Anosov groups}
\author{M. Bestvina, K. Bromberg, A. E. Kent, and C. J. Leininger}
\thanks{The authors gratefully acknowledge support from NSF grants DMS-1308178, DMS-1509171, DMS-1350075, and DMS-1510034.  The third author extends her thanks to the Institute for Advanced Study for its support under NSF grant DMS-1128155 while this work was completed.}

\begin{document}

\maketitle

\begin{abstract}
In this paper we prove that groups as in the title are convex cocompact in the mapping class group.
\end{abstract}

\section{Introduction}

Let $S$ be a finite type surface with negative Euler characteristic, and $\Mod(S)$ its mapping class group.  In \cite{FM:CC}, Farb and Mosher define a notion of {\em convex cocompactness} for a subgroup $G< \Mod(S)$ by requiring that the orbit in Teichm\"uller space be quasi-convex.  More importantly, a subgroup $G < \Mod(S)$ is convex cocompact if and only if in the surface group extension
\[ 1 \to \pi_1(S) \to \Gamma_G \to G \to 1, \]
the group $\Gamma_G$ is hyperbolic when $S$ is closed (\cite{FM:CC,Hamenstadt}), and relatively hyperbolic when $S$ has punctures (\cite{MjSardar}).

The definitions readily imply that if $G$ is convex cocompact, then it is finitely generated and purely pseudo-Anosov, meaning that every infinite order element is pseudo-Anosov.  The next question asks if the converse holds (see \cite{FM:CC}).
\begin{question} \label{Q:motivating question}
If $G<\Mod(S)$ is finitely generated and purely pseudo-Anosov, is it convex cocompact?
\end{question}
This is closely related to Gromov's hyperbolicity question (see \cite[Question 1.1]{bestvina:problems}).  Indeed, if $S$ is closed, and if $G <\Mod(S)$ has a finite $K(G,1)$ and is purely pseudo-Anosov, then $\Gamma_G$ has a finite $K(\Gamma_G,1)$ and no Baumslag-Solitar subgroups (see, e.g.~\cite{KL:survey}), and Gromov's question asks if $\Gamma_G$ is hyperbolic.

Question~\ref{Q:motivating question} seems to be quite difficult, though several classes of finitely generated, purely pseudo-Anosov subgroups have been shown to be convex cocompact; see \cite{KLS:trees,DKL:3-mfd,MT:RAAG,KMT:RAAG}.  In this paper we prove that the additional assumption of being undistorted (that is, quasi-isometrically embedded) in $\Mod(S)$ suffices for convex cocompactness.

\begin{theorem:main} A subgroup $G< \Mod(S)$ is convex cocompact if and only if it is finitely generated, undistorted, and purely pseudo-Anosov.
\end{theorem:main}

In \cite{DT:stable}, Durham and Taylor define a strong form of quasi-convexity they call {\em stability}, and prove that when $\xi(S) \geq 2$ (see \S\ref{S:prelim}), stability in the mapping class group is equivalent to being convex cocompact.\footnote{The assumption $\xi(S) \geq 2$ is necessary, but is missing in \cite{DT:stable}.}
The definition of stability includes the assumption of being undistorted, and it follows easily from the Nielsen-Thurston classification (see \cite{farb:MCG}) and Masur-Minsky distance formula (Theorem~\ref{T:distance formula} below), that stable subgroups must be purely pseudo-Anosov when $\xi(S) \geq 2$. We therefore recover Durham and Taylor's characterization of convex cocompactness as a corollary of the Main Theorem.
\begin{corollary} For a surface $S$ with $\xi(S) \geq 2$, a subgroup $G < \Mod(S)$ is convex cocompact if and only if it is stable in $\Mod(S)$.
\end{corollary}

Koberda-Mangahas-Taylor proved that if $G$ is a subgroup of an {\em admissible} right-angled Artin subgroups of $\Mod(S)$, then $G$ is convex cocompact if and only if $G$ is finitely generated and purely pseudo-Anosov; see \cite{KMT:RAAG} for definitions and the precise statement.  This appealed to the Durham-Taylor stability formulation of convex cocompactness, together with previous work of Mangahas-Taylor \cite{MT:RAAG}.  Admissible right-angled Artin subgroups are in particular undistorted, and so we obtain a generalization of the mapping class group result of \cite{KMT:RAAG}.
\begin{corollary} \label{C:KMT} Suppose $H < \Mod(S)$ is an undistorted, finitely generated right-angled Artin subgroup of the mapping class group, and $G < H$ is any subgroup.  Then $G$ is convex cocompact if and only if $G$ is finitely generated and purely pseudo-Anosov.
\end{corollary}
\begin{proof} The forward implication is immediate from the forward implication of the Main Theorem, so we assume that $G$ is finitely generated and purely pseudo-Anosov, and prove that it must be convex cocompact.  First, the fact that $G$ is purely pseudo-Anosov means that as a subgroup of the right-angled Artin group $H$, $G$ is {\em purely loxodromic}---the centralizer of every nontrivial element is cyclic; see \cite{KMT:RAAG}.  From this and Theorem 1.1 of \cite{KMT:RAAG}, it follows that $G$ is undistorted in $H$.  Since $H$ is undistorted in $\Mod(S)$ by assumption, it follows that $G$ is undistorted in $\Mod(S)$.  By the Main Theorem, $G$ is convex cocompact.
\end{proof}

\noindent {\bf Outline.} 
The forward implication of the Main Theorem is straight forward from the definition, using the thick part of Teichm\"uller space as a model for $\Mod(S)$.  Here we outline a proof of the reverse implication.  We use the following characterization of convex cocompactness in terms of $\C(S)$, the {\em curve graph of $S$}, proved in \cite{KL:CC} and \cite{Hamenstadt}.  
\begin{proposition}  \label{P:ccc is qi in C}
A subgroup $G < \Mod(S)$ is convex cocompact if and only if it is finitely generated and, for any $\alpha \in \C(S)$, the orbit map $g \mapsto g \cdot \alpha$ is a quasi-isometric embedding.
\end{proposition}


A finitely generated subgroup $G < \Mod(S)$ is quasi-isometrically embedded if the orbit map $g \mapsto g \mu$ to the {\em marking graph of $S$}, $G \to \M(S)$, is a quasi-isometric embedding.  Further assuming that $G$ is purely pseudo-Anosov, but not convex cocompact, we ultimately produce an infinite order reducible element of $G$, which is a contradiction.

We begin by proving Proposition~\ref{P:linear projection reducible}, which says that if one can find a sufficiently large group element $g \in G$ and a proper subsurface $Y$ so that the projection distance between $\mu$ and $g \mu$ in the marking graph of $Y$ are linear in word length $|g|$, then $G$ contains a reducible element.  To apply this, we proceed as follows.  

By Proposition~\ref{P:ccc is qi in C}, the assumption that $G$ is not convex cocompact means that there are arbitrarily large group elements $g \in G$ so that the distance between $\mu$ and $g \mu$ in $\C(S)$ grows sub-linearly in $|g|$.  By Theorem~\ref{T:distance formula} (Distance Formula), this means that the sum of other big subsurface projections between $\mu$ and $g  \mu$ must be growing linearly.  Furthermore, by Proposition~\ref{P:overlapping factors} (Overlapping Factors), we may assume that the subsurfaces in this sum {\em overlap} (no two are either disjoint or nested).  

Proposition~\ref{P:subsurface order} (Subsurface Order) provides a natural total order on the subsurfaces appearing in the sum, and appealing to Proposition~\ref{P:Behrstock} (Behrstock Inequality) we show that the path in $\M(S)$ from $\mu$ to $g \cdot \mu$ (coming from a geodesic in the Cayley graph of $G$) is basically forced to traverse the required distance in each of the curve complexes of these subsurfaces one at a time and in order (see Lemma~\ref{L:path traversed in order}).  

These subsurfaces can be divided into maximal intervals of subsurfaces which fill larger, proper subsurfaces of $S$.   Appealing to Proposition~\ref{P:BGI} (Bounded Geodesic Image), we show that the number of these larger subsurfaces is at most the distance in the curve graph of $S$.  Because the curve complexes are traversed one at a time and in order, the marking graphs are also effectively traversed one at a time.  From this, we can efficiently express $g$ as a product of group elements so that each element corresponds to the traversal of one of the marking graphs of these larger subsurfaces.  This is essentially the content of Proposition~\ref{P:linear marking subsurface}.  

Finally, sublinearity of the distance between $\mu$ and $g \mu$ in $\C(S)$ guarantees that one of the elements in the product has length tending to infinity, and projection to the marking graph of the associated subsurface linear in length.  Applying our criterion (Proposition~\ref{P:linear projection reducible}), we obtain a nontrivial reducible element, and hence our desired contradiction.

We note that the general strategy of our proof shares some features with the proof of Theorem 1.1 of \cite{KMT:RAAG} regarding an analogous class of subgroups of right-angled Artin groups, though the techniques are quite different.\\

{\bf Acknowledgements.}  The authors would like to thank Johanna Mangahas for pointing out Corollary~\ref{C:KMT}.

\section{Preliminaries} \label{S:prelim}

By a {\em subsurface of $S$}, we mean a connected, $\pi_1$--injective, properly embedded subsurface $Y \subseteq S$ such that every puncture of $Y$ is a puncture of $S$, and every boundary component is a homotopically essential, nonperipheral closed curve in $S$ (in fact, this latter implies $Y$ is $\pi_1$--injective), and such that $Y$ is not homeomorphic to a $3$--holed sphere.  A {\em curve in $Y$} is a homotopically essential, non-peripheral simple closed curve in $S$.  Subsurfaces and curves will be considered up to isotopy, and we will freely pass between isotopy classes and representatives of the isotopy classes whenever convenient.  Given a subsurface $Y$, let $\xi(Y) = 3g-3+n$, where $g$ is the genus of $Y$ and $n$ is the number of punctures plus the number of boundary components of $Y$.

If $Y$ is not an annulus, the {\em curve graph of $Y$} is the simplicial graph, $\C(Y)$, whose vertices are curves in $Y$ and whose edges are pairs of distinct curves that can be realized with minimal intersection in $Y$ (that is, pairwise disjoint if $Y$ is not a four-punctured sphere or once-punctured torus and intersecting twice or once, in these two cases, respectively).  If $Y$ is an annulus, $\C(Y)$ is defined as follows.  Let $\widetilde Y$ be the natural compactification of the cover of $S$ in which $Y$ lifts so that the inclusion is a homotopy equivalence. The vertices of $\C(Y)$ are (isotopy classes of) arcs connecting the distinct boundary components of $\widetilde Y$ and edges are pairs of arcs that can be realized with disjoint interiors.  For any two vertices $\alpha,\alpha' \in \C(Y)$, the distance between $\alpha$ and $\alpha'$ in $\C(Y)$ is defined to be the minimal length (number of edges) of any edge-path between $\alpha$ and $\alpha'$ in $\C(Y)$.  A geodesic is any minimal length edge-path.  According to \cite{MM:CC1}, $\C(Y)$ is a Gromov hyperbolic, geodesic metric space.

We say that two proper subsurfaces $Y,Z \subset S$ {\em overlap} if they cannot be realized disjointly and neither can be realized as a subsurface of the other.  In this case, we write $Y \pitchfork Z$.  A curve $\alpha$ {\em cuts} a subsurface $Y$ if $\alpha$ cannot be realized disjoint from $Y$, and in this case we similarly write $\alpha \pitchfork Y$.  If $\alpha$ is a curve and $Y$ is a subsurface with $\alpha \pitchfork Y$, then the {\em projection of $\alpha$ to $Y$}, $\pi_Y(\alpha) \subset \C(Y)$ is defined as follows; see \cite{MM:CC2}.  If $Y$ is an annulus, then $\pi_Y(\alpha)$ is the union of the arcs of $\widetilde Y$ which are (closures of) arcs of the preimage of $\alpha$ in $\widetilde Y$ with endpoints on distinct boundary components.  If $Y$ is not an annulus, then realize $\alpha$ and $Y$ so that they intersect minimally, and let $\alpha'$ be any arc (or simple closed curve) of $\alpha \cap Y$.  There is at least one component of the regular neighborhood of $\alpha' \cup \partial Y$ which is essential in $Y$, and we let $\pi_Y(\alpha)$ denote the union of all curves in $Y$ so obtained  (over all choices of arc $\alpha'$).  If $\alpha \not \pitchfork Y$, we define $\pi_Y(\alpha) = \emptyset$.

If $Y$ is not an annulus, a marking $\mu$ on $Y$ is maximal set of
pairwise disjoint curves $b$ in $Y$ (i.e.~a {\em pants decomposition})
called the {\em base of $\mu$}, together with a diameter $1$ subset
$t_\alpha \subset \C(Y_\alpha)$ for each $\alpha \in b$, where
$Y_\alpha$ is the annular neighborhood of $\alpha$.  The subset $t_\alpha$ is called a {\em transversal for $\alpha$}. If $Y$ is an annulus, then a marking is just an vertex of $\C(Y)$; see \cite{MM:CC2,BKMM}.  Markings are considered up to isotopy, and the set of markings on $Y$ are the vertices of a connected graph $\mathcal M(Y)$ called the {\em marking graph of $Y$}.  Edges correspond to markings that differ by {\em elementary moves}.  We will not need the specifics of this definition, instead we note that $\Mod(S)$ acts on $\mathcal M(S)$ with the following key properties; \cite{MM:CC2}.
\begin{proposition} [Mapping Class Group Marking Graph] \label{P:MCGMG}
For any finite generating set of $\Mod(S)$ and element $\mu \in \mathcal M(S)$, the orbit map $\Mod(S) \to \mathcal M(S)$, defined by $g \mapsto g \cdot \mu$, is a quasi-isometric embedding.
\end{proposition}

Markings can also be projected to either curve complexes or marking graphs of subsurfaces.  Given $Y \subseteq Z \subseteq S$, and any $\mu$ in $\mathcal M(Z)$ we write $\pi_Y(\mu) \subset \C(Y)$ and $\pi_{\mathcal M(Y)}(\mu) \subset \mathcal M(Y)$ for these projections.  The projection $\pi_Y(\mu)$ is defined as the union of the projections of all base curves to $Y$, unless $Y$ is an annulus whose core curve is itself one of the base curves $\alpha \in b$.  In this latter situation, $Y = Y_\alpha$, and $\pi_Y(\mu)$ is defined as $t_\alpha \in \C(Y)$, the transversal of $\alpha$.  The projection to $\mathcal M(Y)$ is defined by an inductive procedure, making several choices, then taking the union over all choices.   Again, we will not need the specifics of these projections, but instead we list here various facts that will be important for us.

We begin with the following; see \cite{MM:CC2,Behrstock,BKMM}.

\begin{proposition} [Bounded Diameter Projection] \label{P:bounded diameter projection} There is a constant $\delta > 0$, depending
  on $S$, so that if $\mu$ is a marking or curve on $Z \subseteq S$ and
  $Y \subseteq Z$, then $\pi_Y(\mu)$ and $\pi_{\mathcal M(Y)}(\mu)$ has
  diameter at most $\delta$. 
  \end{proposition}

For any two curves or markings $\mu_1,\mu_2$ in $Z \subseteq S$ and $Y \subseteq Z$ (with $\mu_1,\mu_2 \pitchfork Y$ if $\mu_1,\mu_2$ are curves), we define
\[ d_Y(\mu_1,\mu_2) = \diam_{\C(Y)}(\pi_Y(\mu_1) \cup \pi_Y(\mu_2)).\]
Similarly, for $\mu_1,\mu_2 \in \mathcal M(Z)$, define
\[ d_{\mathcal M(Y)}(\mu_1,\mu_2) = \diam_{\mathcal
  M(Y)}(\pi_{\mathcal M(Y)}(\mu_1) \cup \pi_{\mathcal
  M(Y)}(\mu_2)).\]
This particular choice of distance makes the triangle inequality hold whenever the relevant projections are nonempty. Along with Proposition~\ref{P:bounded diameter projection}, another basic fact is that projections are Lipschitz.
\begin{proposition} [Lipschitz projection] \label{P:Lipschitz projections}
There exists a constant $\delta'$, depending on $S$, so that for all $Y \subseteq Z \subseteq S$ and $\mu_1,\mu_2 \in  \mathcal M(Z)$,
\[ d_Y(\mu_1,\mu_2), d_{\mathcal M(Y)}(\mu_1,\mu_2) \leq \delta' d_{\mathcal M(Z)}(\mu_1,\mu_2). \]
\end{proposition}

A strong boundedness property of projections is the following, due to Masur and Minsky \cite{MM:CC2}.

\begin{proposition} [Bounded Geodesic Image] \label{P:BGI} There exists $M > 0$, depending on $S$, so that for any two curves or markings $\mu_1,\mu_2$ on $Z \subseteq S$ and proper subsurface $Y \subsetneq Z$ (with $\mu_1,\mu_2 \pitchfork Y$ if $\mu_1,\mu_2$ are curves), if $d_Y(\mu_1,\mu_2) \geq M$, then any geodesic between $\mu_1,\mu_2$ in $\C(Z)$ must have a vertex $\alpha$ so that $\pi_Y(\alpha) = \emptyset$. 
\end{proposition}

Another important important bound for projections is the following, due to Behrstock \cite{Behrstock} (see also \cite{Mangahas}).

\begin{proposition} [Behrstock Inequality] \label{P:Behrstock} There exists $B > 0$ so that if $Y,Z \subsetneq S$ and $\gamma \in \C(S)$ is any curve with $\gamma \pitchfork Y,Z$, then
\[ \min \{ d_Y(\gamma,\partial Z),d_Z(\gamma,\partial Y) \} < B. \]
\end{proposition}
Given two markings $\mu_1,\mu_2 \in \M(S)$ and $\beta > 0$, define
\[ \Omega^\circ_\beta(\mu_1,\mu_2) = \{Y \subsetneq S \mid d_Y(\mu_1,\mu_2) \geq \beta \}\]
and
\[ \Omega_\beta (\mu_1,\mu_2) = \{ Z \subseteq S \mid Z \mbox{ is filled by subsurfaces } Y \subseteq Z \mbox{ with } Y \in \Omega_\beta^\circ(\mu_1,\mu_2) \}.\]
Here, we say that a subsurface $Z$ is {\em filled} by a collection of subsurfaces $\{Y_\alpha \subseteq Z\}_\alpha$ if either $Z$ is an annulus and $\{Y_\alpha\}_\alpha = \{Z\}$, or $Z$ is not an annulus and for every curve $\gamma \in \C(Z)$, there exists $\alpha$ so that $\gamma \pitchfork Y_\alpha$.
The following is a straightforward consequence of Proposition~\ref{P:Behrstock} (Behrstock Inequality) proven in \cite{BKMM} (see also \cite{CLM:RAAGS}).
\begin{proposition} [Subsurface Order] \label{P:subsurface order} Given two markings $\mu_1,\mu_2$ and $\beta > 2B$ (with $B > 0$ from Proposition~\ref{P:Behrstock} (Behrstock Inequality)) there is a partial order on $\Omega^\circ_\beta(\mu_1,\mu_2)$ such that $Y,Z \in \Omega^\circ_\beta(\mu_1,\mu_2)$ are comparable if and only if $Y \pitchfork Z$.  In this case, the following are equivalent
\[ \begin{array}{rlrlrl}
(1) & Y < Z,  \quad & (2) & d_Y(\mu_1,\partial Z) \geq B, \quad & (3) & d_Z(\mu_2,\partial Y) \geq B, \quad \\
& & (4) & d_Z(\mu_1,\partial Y) < B, & (5) & d_Y(\mu_2,\partial Z) < B. \end{array} \]
\end{proposition}

One final fact about projection distances is the following theorem.  
Given $\beta > 0$ and $x \in \mathbb R$ we write
\[ \{\!\{x \}\!\}_\beta  = \left\{ \begin{array}{ll} x & \mbox{ if } x \geq \beta\\ 0 & \mbox{ otherwise.} \end{array} \right. \]
\begin{theorem} [Distance Formula] \label{T:distance formula}  Given any $\beta > 0$ sufficiently large there exists $\kappa \geq 1$ with the following property.  If $\mu_1,\mu_2 \in \M(S)$ and $Z \subseteq S$, then
\[ \frac1\kappa d_{\mathcal M(Z)}(\mu_1,\mu_2) \leq \sum_{Y \subseteq Z} \{\!\{d_Y(\mu_1,\mu_2)\}\!\}_\beta \leq \kappa d_{\mathcal M(Z)}(\mu_1,\mu_2) ,\]
for all $\mu_1,\mu_2$ such that either the sum in the middle is nonzero or $d_{\mathcal M(Z)}(\mu_1,\mu_2) \geq \kappa$.  Furthermore, when $Z = S$, there are only finitely many $\Mod(S)$--orbits of pairs $(\mu_1,\mu_2)$ (under the diagonal action) in which the middle term is zero.
\end{theorem}
The original distance formula, due to Masur-Minsky \cite{MM:CC2}, has an additive error (in addition to the multiplicative error $\kappa$) instead of the conditional validity of the inequality, which is more useful for our purposes.  Since the distances are all integers, the version here follows easily from the original one.   The original formula was also stated only for $S$ instead of for subsurfaces $Z \subseteq S$.  The variant for a subsurface follows from the ``coarse transitivity'' of iterated projections for nested subsurfaces (see \cite[Lemma 2.12]{BKMM}).

The following is an easy consequence of Theorem~\ref{T:distance formula} (Distance Formula).
\begin{corollary}\label{C:finitely many smallest-big subsurfaces} Suppose $\beta \geq 1$ (sufficiently large) and $\kappa \geq 1$ are as in Theorem~\ref{T:distance formula} (Distance Formula).  Then for all $\mu_1,\mu_2 \in \M(S)$, the set $\Omega_\beta(\mu_1,\mu_2)$ is finite.  Furthermore, if $d_{\M(Z)}(\mu_1,\mu_2) \geq \kappa \beta$, then there is a subsurface $W \subseteq  Z$ such that $W \in \Omega_\beta(\mu_1,\mu_2)$ and
\[ \frac{1}{\xi(Z)\kappa^2} d_{\M(Z)}(\mu_1,\mu_2) \leq d_{\M(W)}(\mu_1,\mu_2) \leq \kappa^2 d_{\M(Z)}(\mu_1,\mu_2) .\]
\end{corollary}
\begin{proof} Since $d_{\M(S)}(\mu_1,\mu_2)$ is a finite number for any $\mu_1,\mu_2 \in \M(S)$, it follows from Theorem~\ref{T:distance formula} (Distance Formula) that there are only finitely many $Y$ with $d_Y(\mu_1,\mu_2) \geq \beta$; that is, $\Omega^\circ_\beta(\mu_1,\mu_2)$ is finite.  The finite set of subsurfaces $Z \subseteq S$ filled by the subsurfaces in $\Omega^\circ_\beta(\mu_1,\mu_2)$ is exactly $\Omega_\beta(\mu_1,\mu_2)$, proving the first statement.

Next, list the (finitely many) subsurfaces of $Z$ in $\Omega_\beta^\circ(\mu_1,\mu_2)$:
\[ \{Y \subseteq Z \mid Y \in \Omega_\beta^\circ(\mu_1,\mu_2)\} = \{Y_1,\ldots,Y_k\}.\]
Since $d_{\M(Z)}(\mu_1,\mu_2) \geq \kappa \beta \geq \kappa$, Theorem~\ref{T:distance formula} (Distance Formula) implies
\[ \sum_{j=1}^k d_{Y_j}(\mu_1,\mu_2) \geq \frac{1}{\kappa} d_{\M(Z)}(\mu_1,\mu_2),\]
so $\{Y_1,\ldots,Y_k\}$ is nonempty.  If these subsurfaces fill $Z$, then $Z \in \Omega_\beta(\mu_1,\mu_2)$ and we are done.  Otherwise, we let $W_1,\ldots,W_r$ be the component subsurfaces of $Z$ filled by $Y_1,\ldots,Y_k$, and note that $r < \xi(Z)$. Iteratively applying Theorem~\ref{T:distance formula} (Distance Formula) we have
\[ \sum_{i=1}^r d_{\M(W_i)}(\mu_1,\mu_2) \leq \kappa \sum_{j=1}^k d_{Y_j}(\mu_1,\mu_2)  \leq \kappa^2 d_{\M(Z)}(\mu_1,\mu_2).\]
and
\[ \sum_{i=1}^r d_{\M(W_i)}(\mu_1,\mu_2) \geq \frac{1}{\kappa} \sum_{j=1}^k d_{Y_j}(\mu_1,\mu_2)  \geq \frac{1}{\kappa^2} d_{\M(Z)}(\mu_1,\mu_2).\]
Let $i \in \{1,\ldots,r\}$ be the such that the term $d_{\M(W_i)}(\mu_1,\mu_2)$ in the sum above is largest, and set $W = W_i$ so that
\[ d_{\M(W)}(\mu_1,\mu_2) \geq \frac{1}{r}  \sum_{i=1}^r d_{\M(W_i)}(\mu_1,\mu_2) \geq \frac{1}{\xi(Z)}  \sum_{i=1}^r d_{\M(W_i)}(\mu_1,\mu_2).\]
Since $W \in \Omega_\beta(\mu_1,\mu_2)$, these inequalities complete the proof.
\end{proof}

In \cite{BBF}, the first and second authors, with Fujiwara, construct a partition of the set of subsurfaces into finitely many subsets that we will refer to as {\em BBF factors}.  The key property of a BBF factor is stated in the following.
\begin{proposition} [Overlapping Factors] \label{P:overlapping factors} 
For any surface, a BBF factor $\bY$ has the property that either $\bY = \{S\}$, or else, for all $Y \neq Y' \in \bY$, $Y \pitchfork Y'$.
\end{proposition}

\section{A criterion for reducibility.}

\begin{proposition}[Linear projection reducibility] \label{P:linear projection reducible}
Suppose $G < \Mod(S)$ is finitely generated and let $|g|$ denote the
word length of $g \in G$ with respect to a finite generating set and
let $\mu$ be a fixed marking.  Then for any $c >0$ there exists $R >
0$ so that if $|g| \geq R$ and if there exists a proper subsurface $Z \subset S$
with
\[ d_{\M(Z)}(\mu,g \mu) \geq c |g| \]
then $G$ contains a nontrivial reducible element.
\end{proposition}
The proposition basically says that if there exists arbitrarily large group elements $g \in G$ so that on a proper subsurface $Z$, $d_{\M(Z)}(\mu,g\mu)$ is (at least) linear in $|g|$, then $G$ contains a nontrivial reducible element.  
Before we give the proof, we sketch the idea under the stronger assumption that the projection $d_Z(\mu,g\mu)$ is linear in $|g|$.

Finite generation guarantees that there are only finitely many ``big projections'' among uniformly bounded length group elements.  Considering the geodesic in the Cayley graph from the identity to $g$ as being a concatenation of uniformly bounded length group elements, we see that a definite percentage of these must contribute to the linear growth of the distance in $\C(Z)$.  Each contribution comes from a translate, by an initial segment of the geodesic, of one of the finitely many big projections.  The pigeonhole principle ensures that two distinct initial segments of the geodesic are translating the same subsurface, and hence the composition of one with the inverse of the other fixes that subsurface, and is hence reducible.  

The case of marking graph projections is similar.  If $d_{\M(Z)}(\mu,g\mu)$ is linear in $|g|$, then we pass to a minimal complexity subsurface $W \subseteq Z$ for which $d_{\M(W)}(\mu,g\mu)$ is also linear in $|g|$.  The required finiteness needed to apply the pigeonhole principle follows from the minimal complexity of $W$, appealing to Corollary~\ref{C:finitely many smallest-big subsurfaces}.

\begin{proof}  Suppose that there exists $c' > 0$ such that for all $R >0$, there exists $g \in G$ with $|g| \geq R$ and a proper subsurface $Z \subset S$ with
\[ d_{\M(Z)}(\mu,g\mu) \geq c'|g|.\]
If there is no such $c'$, then the proposition holds vacuously.

Next, consider the smallest integer $\xi_0$ such that for some $c >0$ the following holds.  For all $R > 0$ there exists $g \in G$ with $|g| > R$ and a subsurface $Z$ with $\xi(Z) = \xi_0$ and
\[ d_{\M(Z)}(\mu,g\mu) \geq c |g|.\]
The first paragraph guarantees that $\xi_0$ exists.  Indeed, an upper bound for $\xi_0$ is obtained as the minimum of $\xi(Z)$ such that there exists $g \in G$ with $|g| \geq R$ and $d_{\M(Z)}(\mu,g\mu) \geq c'|g|$.  Fix this minimal $\xi_0$ and the associated $c > 0$, and let 
\[ \mathcal G = \{g \in G \mid d_{\M(Z)}(\mu,g\mu) \geq c|g| \mbox{ for some } Z \mbox{ with } \xi(Z) = \xi_0\}.\]
By assumption, $\mathcal G$ is an infinite set (and in particular, there exists $g \in \mathcal G$ with $|g|$ as large as we like).
Given $g \in \mathcal G$, let $Z(g)$ be a subsurface with $d_{\M(Z(g))}(\mu,g\mu) \geq c|g|$ and $\xi(Z(g)) = \xi_0$.

For any $L > 0$ let
\[ B(L) = \max \{ d_{\M(Z)}(\mu,g\mu) \mid g \in G, \, |g | \leq L, \mbox{ and } Z \subset S \mbox{ is a proper subsurface} \}.\]
Fix any $g \in \mathcal G$, let $Z = Z(g)$, and let $g = g_0g_1 \cdots g_n$ such that $g_0 = id$, $|g_j| = L$ for $1 \leq j  < n$, $|g_n| \leq L$, and
\[ \sum_{j=1}^n |g_j| = |g|.\]
(Note that $n$ depends on $g$.)
For each $j = 0,\ldots,n$, we also write $h_j = g_0 \cdots g_j$, so that $h_j = h_{j-1}g_j$ for all $1 \leq j \leq n$.
Partition the set $\{1,\ldots,n\}$ into two subsets:
\[ J_\ell^L = J_\ell^L(g) = \{j \mid d_{\M(Z)}(h_{j-1}\mu,h_j\mu) \geq \frac{c}{2} |g_j| \},\]
and
\[ J_s^L = J_s^L(g) = \{j \mid d_{\M(Z)}(h_{j-1}\mu,h_j\mu) < \frac{c}{2} |g_j| \}.\]
Then since $|g_j| \leq L$ for all $j$, our assumptions and the triangle inequality implies
\begin{eqnarray*}
c|g| & \leq & d_{\M(Z)}(\mu,g\mu) \\
& \leq & \sum_{j \in J_s^L}d_{\M(Z)}(h_{j-1}\mu,h_j\mu) + \sum_{j \in J_\ell^L} d_{\M(Z)}(h_{j-1}\mu,h_j\mu) \\
& \leq & \frac{c}{2} \sum_{j \in J_s^L} |g_j| + \sum_{j \in J_\ell^L} d_{\M(h_{j-1}^{-1}Z)}(\mu,g_j\mu) \\
& \leq & \frac{c}{2} |g| + B(L)|J_\ell^L|.
\end{eqnarray*}
Therefore, for all $g \in \mathcal G$ and $L > 0$ we have have
\begin{equation} \label{E:min wins} |J_\ell^L(g)| \geq \frac{c}{2B(L)}|g|.\end{equation}

Now let $\beta,\kappa \geq 1$ be as in Corollary~\ref{C:finitely many smallest-big subsurfaces}.\\

\noindent
{\bf Claim.} For all $L$ sufficiently large and $g \in \mathcal G$ with $|g| > L$, if we write $g = g_0g_1 \cdots g_n$ as above and let $j \in J_\ell^L(g) - \{n\}$, then
\[ h_{j-1}^{-1}Z(g) \in \Omega_\beta(\mu,g_j\mu). \]
\begin{proof}[Proof of Claim.]
Observe that if $L \geq \frac{2\kappa\beta}{c}$, then for $j \in J_\ell^L(g) - \{n\}$, $|g_j| = L$, and
\[d_{\M(h_{j-1}^{-1}Z(g))}(\mu,g_j\mu) = d_{\M(Z(g))} (h_{j-1}\mu,h_j\mu) \geq \frac{c}{2} |g_j| = \frac{cL}{2} \geq \kappa \beta. \]
Therefore, either $h_{j-1}^{-1}Z(g) \in \Omega_\beta(\mu,g_j\mu)$ and we are done, or else Corollary~\ref{C:finitely many smallest-big subsurfaces} implies that there is a proper subsurface $W \subset h_{j-1}^{-1}Z(g)$ so that $W \in \Omega_\beta(\mu,g_j\mu)$ and
\[ d_{\M(W)}(\mu,g_j\mu) \geq \frac{1}{\xi(Z(g))\kappa^2} d_{\M(h_{j-1}^{-1}Z(g))}(\mu,g_j\mu) \geq \frac{c}{2\xi(Z(g))\kappa^2}|g_j|.\]
But if there are arbitrarily large $L > 0$, $g \in \mathcal G$ with $|g| > L$, and $j \in J_\ell^L(g) - \{n\}$ for which this inequality holds, then the fact that $\xi(W) < \xi(h_{j-1}^{-1}(Z(g))) = \xi_0$ contradicts our minimality assumption on $\xi_0$ since $\frac{c}{2\xi(Z(g))\kappa^2}$ is constant.  This proves the claim.
\end{proof}

To complete the proof, let $L > 0$ be large enough for the claim to hold.  By (\ref{E:min wins}), we may choose an $R>0$ such that if $g \in \mathcal G$ with $|g|>R$ then $|J_\ell^L(g)|$ is as large as we like.  In particular, we choose $R$ large enough so that
\[ \left| \bigcup_{\{ h \in G \mid |h| = L\}} \Omega_\beta(\mu,h\mu) \right| < |J_\ell^L(g) - \{n\}| \]
By the claim, $h_{j-1}^{-1}Z(g) \in \Omega_\beta(\mu,g_j\mu)$ for all $j \in J_\ell^L(g) - \{n\}$, and since $|g_j| = L$, the Pigeonhole Principle implies that some subsurface from the set
\[ \{h_{j-1}^{-1}Z(g) \mid j \in J_\ell^L(g) - \{n\} \}\]
must be repeated in this listing.  That is,
\[ h_{j-1}^{-1}(Z(g)) = h_{i-1}^{-1}(Z(g))\]
for some $i, j \in J_\ell^L(g) - \{n\}$ with $i > j$.  But $h_{i-1}h_{j-1}^{-1} \in G$ 
sends $Z(g)$ to itself, and is nontrivial since $|h_{i-1}| = (i-1)L \neq (j-1)L = |h_{j-1}|$.
Therefore, $h_{i-1}h_{j-1}^{-1}$ is a nontrivial reducible element of $G$.
\end{proof}

\section{Linear factors}

To prove the Main Theorem, we will show that $d_S(\mu,g\mu)$ is larger than some linear function of $|g|$, and then apply Proposition~\ref{P:ccc is qi in C}.  The proof is by contradiction, and so we will need to know what happens in an undistorted subgroup when $d_S(\mu,g\mu)$ is not linear in $|g|$.  
The main technical proposition we will use is the following.

\begin{proposition}[Linearly summing projections] \label{P:linear marking subsurface}
Given $G < \Mod(S)$, an undistorted subgroup with a fixed finite generating set, and a marking $\mu$, there exist $K,C > 0$ with the following property.  

For all $g \in G$ with $|g| > C$, either $|g| \leq K d_S(\mu,g\mu)$ or else there exist proper subsurfaces $Z_1,\ldots,Z_k \subset S$ and $g_1,\ldots,g_k \in G$ such that
\begin{enumerate}
\item[{\em (i)}] $d_S(\mu,g\mu) \geq k$,
\item[{\em (ii)}] $g = g_1g_2 \cdots g_k$ with $|g| = |g_1| + \cdots + |g_k|$, 
\item[{\em (iii)}] $\displaystyle{|g| \leq K \sum_{j=1}^k d_{\mathcal M(Z_j)}(\mu,g_j \mu)}$, and
\item[{\em (iv)}] $d_{\mathcal M(Z_j)}(\mu,g_j\mu) \leq K|g_j|$, for all $j = 1,\ldots,k$.
\end{enumerate}
\end{proposition}

According to Proposition~\ref{P:Lipschitz projections} (Lipschitz projection), (iv) is automatic as soon as $K$ is sufficiently large, and so we focus on (i) - (iii).  The proof requires two constructions and several lemmas.  Fix an undistorted subgroup $G < \Mod(S)$, a finite generating set, and a marking $\mu$ for the remainder of this section.

\begin{lemma} \label{L:linear factor}
For $\beta > 0$ sufficiently large, there exists $K',C' > 0$ such that if $g \in G$ with $|g| > C'$ then
\[ |g| \leq K' \sum_{Y \in \bY} \{\!\{ d_Y(\mu,g \mu) \}\!\}_\beta \]
for some BBF factor $\bY$.
\end{lemma}
\begin{proof}  The factors form a {\em finite} partition of the set of subsurfaces, so is is immediate from Theorem~\ref{T:distance formula} (Distance Formula).
\end{proof}

By Proposition~\ref{P:Lipschitz projections} (Lipschitz projection), there exists $b > 0$ so that for any subsurface $Z$ we have
\[ d_{\mathcal M(Z)}(\mu,x \mu) \leq b \mbox{ and } d_Z(\mu,x\mu) \leq b\]
for each of our finitely many generators $x$ of $G$.  We assume (as we may) that $b \geq B$ from Propositions~\ref{P:Behrstock} (Behrstock Inequality) and \ref{P:subsurface order} (Subsurface Order).  Fix any $\beta > 5b + M + 3 \delta > 2B$, where $M$ is the constant from Proposition~\ref{P:BGI} (Bounded Geodesic Image) and $\delta$ is the constant from Proposition~\ref{P:bounded diameter projection} (Bounded Diameter Projections), and set $\beta_0 = \beta + 5b$.  

For any $g \in G$ and factor $\bY \neq \{S\}$, let $Y_1,\ldots,Y_n$ be the set of all subsurfaces in $\bY$ such that $d_{Y_i}(\mu,g\mu) \geq \beta_0$.  Further assume they are ordered as in Proposition~\ref{P:subsurface order} (Subsurface Order) with $Y_i < Y_j$ for all $i < j$ (see also Proposition~\ref{P:overlapping factors} (Overlapping factors)).  Consider a geodesic in (the Cayley graph of) $G$ from $id \in G$ to $g \in G$.  Consecutive group elements along the geodesic differ by one of the generators, and applying these elements to $\mu$ gives a discrete path of markings in $\mathcal M(S)$, which in turn project to discrete paths in each curve graph $\C(Y_i)$, starting at $\pi_{Y_i}(\mu)$ and ending at $\pi_{Y_i}(g\mu)$.  

Roughly speaking, the next lemma says the paths in $\C(Y_i)$ respect the ordering $Y_1 < Y_2 < \cdots  < Y_n$, meaning that the projection to $\C(Y_i)$ cannot make progress from $\pi_{Y_i}(\mu)$ toward $\pi_{Y_i}(g \mu)$ until the projection to $\C(Y_{i-1})$ is sufficiently close to $\pi_{Y_{i-1}}(g \mu)$.  This is a straightforward consequence of Propositions~\ref{P:Behrstock} (Behrstock Inequality) and \ref{P:subsurface order} (Subsurface Order), and is reflected in the quasi-tree behavior proved in \cite{BBF}.

To make this precise, we first define a {\em prefix} of $g \in G$ to be an element $g' \in G$ so that
$g = g'g''$ and $|g| = |g'| + |g''|$.  If $g'$ is a prefix of $g$, we write $g' \preceq g$ (and $g' \prec g$ if $g' \neq g$).  Note that $\preceq$ is a partial order on the prefixes of $g$, and any maximally ordered chain of prefixes are the vertices of a geodesic in $G$.  Fix such a geodesic, and for each $1 \leq j < n$, let $g_j'$ be the longest prefix of that geodesic such that
\[ d_{Y_j}(g_j'\mu,g \mu) \geq 2b. \]
To avoid special cases, we also let $g_0' = id$ and $g_n'= g$.

\begin{lemma} \label{L:path traversed in order} For all $0 \leq i < j \leq n$, we have $g_i' \prec g_j'$ and if $i  < \ell \leq j$, then
\[ |d_{Y_\ell}(\mu,g\mu) - d_{Y_\ell}(g_i'\mu,g_j'\mu)| < 5b.\]
\end{lemma}
\begin{proof}  For any $1 \leq i < n$, if $x$ is a generator so that $g_i'x$ is also a prefix of the geodesic for $g$, then by maximality of the length of $g_i'$, we have $d_{Y_i}(g_i'x\mu,g\mu) < 2b$.  On the other hand, 
\[ d_{Y_i}(g_i'\mu,g_i'x\mu) = d_{(g_i')^{-1}Y_i}(\mu,x\mu) \leq b. \]
Thus,
\begin{equation} \label{E:i close to i exit} d_{Y_i}(g_i'\mu,g\mu) \leq d_{Y_i}(g_i'\mu,g_i'x\mu) + d_{Y_i}(g_i'x\mu,g\mu) < b + 2b = 3b.
\end{equation}
Inequality (\ref{E:i close to i exit}) also clearly holds for $i = n$ since $g_n' = g$.

Now suppose $1 \leq i < j \leq n$.  Since $Y_i < Y_j$, we have $d_{Y_i}(g\mu,\partial Y_j) < B \leq b$ and $d_{Y_j}(\mu,\partial Y_i) < B \leq b$.
Therefore
\[ d_{Y_i}(g_i'\mu,\partial Y_j) \geq d_{Y_i}(g_i'\mu,g\mu) - d_{Y_i}(g\mu,\partial Y_j) > 2b - b = b \geq B,\]
and by Proposition~\ref{P:Behrstock} (Behrstock Inequality), we have $d_{Y_j}(g_i'\mu,\partial Y_i) < B \leq b$.  
Hence
\begin{equation} \label{E:i close to j entrance}
d_{Y_j}(\mu,g_i'\mu) \leq d_{Y_j}(\mu,\partial Y_i) + d_{Y_j}(\partial Y_i,g_i'\mu) < 2b,
\end{equation}
and so
\[ d_{Y_j}(g_i'\mu,g\mu) \geq d_{Y_j}(\mu,g\mu) - d_{Y_j}(\mu,g_i'\mu)  \geq \beta_0 - 2b > \beta \geq 5b.\]
If $x$ is a generator so $g_i'x$ is a prefix of the geodesic for $g$, then
\[ d_{Y_j}(g_i'x\mu,g\mu) \geq d_{Y_j}(g_i'\mu,g\mu) - d_{Y_j}(g_i'x\mu,g_i'\mu)  \geq 5b-b \geq 4b.\]
It follows that $g_i' \prec g_j'$.
Furthermore, by the maximality of the length of $g_i'$, $d_{Y_i}(g_j'\mu,g\mu) < 2b$.  
From this and (\ref{E:i close to i exit}), we see that if $1 \leq i \leq j \leq n$, then
\[ d_{Y_i}(g_j'\mu,g\mu) \leq 3b.\]

Now assuming $1 \leq i < \ell \leq j \leq n$, this inequality and (\ref{E:i close to j entrance}) implies
\begin{eqnarray*} |d_{Y_\ell}(\mu,g\mu) - d_{Y_\ell}(g_i'\mu,g_j'\mu)| &  \leq & d_{Y_\ell}(\mu,g_i'\mu) + d_{Y_\ell}(g_j' \mu,g \mu) <  2b + 3b = 5b.
 \end{eqnarray*}
This also clearly holds for $i = 0$.
\end{proof}

We continue to assume $Y_1,\ldots,Y_n$ are the subsurfaces in a BBF factor $\bY \neq \{S\}$ with $d_{Y_i}(\mu,g\mu) \geq \beta_0 = \beta + 5b$.
For all $1 \leq i \leq j \leq n$, let $Y_{ij}$ be the subsurface filled by $Y_i,\dots, Y_j$. 
We then choose $0 = i_0 <  i_1 < \cdots< i_k = n$ such that $Z_j' = Y_{(i_{j-1}+1)(i_{j})}$ is a proper subsurface of $S$, but $Y_{(i_{j-1}+1)(i_j+1)} = S$.


\begin{lemma} \label{L:containers bound main distance}
If $Z_1',\dots, Z_k'$ are as above, then $d_S(\mu,g\mu) \ge k$.
\end{lemma}

\begin{proof} Fix a geodesic $\gamma$ in $\C(S)$ between $\alpha \in \mu$ and $\alpha' \in g\mu$ so that $\alpha \pitchfork Y_1$ and $\alpha' \pitchfork Y_n$.  For all $1 \leq i \leq n$, one deduces from Proposition~\ref{P:subsurface order} (Subsurface Order) that $\alpha,\alpha' \pitchfork Y_i$, and from Proposition~\ref{P:bounded diameter projection} (Bounded Diameter Projection), that $d_{Y_i}(\alpha,\alpha') \geq \beta_0 -2 \delta$.  By Proposition~\ref{P:BGI} (Bounded Geodesic Image), there is at least one curve in $\gamma$ that is disjoint from $Y_i$, and we let $\alpha_i$ be the largest, as ordered by the appearance in $\gamma$ from $\alpha$ to $\alpha'$.

We claim that if $i<j$, then $\alpha_i \le \alpha_j$.  To see this, suppose $\alpha_j < \alpha_i$.  Then by Proposition~\ref{P:BGI}~(Bounded Geodesic Image), $d_{Y_j}(\alpha_i,\alpha') \leq M$ while Proposition~\ref{P:subsurface order}~(Subsurface Order), guarantees that $d_{Y_j}(\alpha,\partial Y_i) < B$, and hence
\begin{eqnarray*} d_{Y_j}(\alpha_i,\partial Y_i) & \geq & d_{Y_j}(\alpha,\alpha')  - d_{Y_j}(\alpha,\partial Y_i) - d_{Y_j}(\alpha_i,\alpha') \\
& > & \beta_0 - 2\delta - M - B \geq \delta \end{eqnarray*}
But then $\partial Y_i$ and $\alpha_i$ must intersect (otherwise their projections would have distance at most $\delta$ by Proposition~\ref{P:bounded diameter projection} (Bounded Diameter Projection)) a contradiction.

It follows that if $i < j$ and $\alpha_i = \alpha_j$ then for all $i< \ell<j$ we have $\alpha_i = \alpha_\ell = \alpha_j$. Since the surfaces $Y_{i_{j-1} + 1}, \dots, Y_{{i_j} +1}$ fill $S$ this implies that for all $j = 1,\ldots,k-1$, $\alpha_{i_{j-1}+1} \neq \alpha_{i_j+1}$ so there must be at least $k-1$ distinct $\alpha_i$.  Since $\alpha_i \neq \alpha,\alpha'$ for all $i$, $d_S(\mu,g\mu) \geq d_S(\alpha,\alpha') \geq k$. 
\end{proof}

\begin{proof}[Proof of Proposition~\ref{P:linear marking subsurface}.]   Continue to assume that $d_Z(\mu,x \mu),d_{\mathcal M(Z)}(\mu,x\mu) \leq b$ for each generator $x$ and suppose $\beta > 5b + M + 3 \delta > 2B$ is sufficiently large for Theorem~\ref{T:distance formula} (Distance Formula) to hold and set $\beta_0 = \beta + 5b$.   Let $\kappa > 0$ be the constant from that theorem applied to $\beta$.  Let $C',K'$ be the constants from Lemma~\ref{L:linear factor}, and let $C = C'$ and $K = 2 \kappa K'$. 

For any $|g| > C = C'$, if $|g| \leq K' d_S(\mu,g\mu)$, then $|g| \leq K d_S(\mu,g\mu)$.  Since this is one of the possible conclusions of the proposition, for the remainder of the proof we only consider elements $g \in G$ with $|g| > K' d_S(\mu,g\mu)$ and $|g| > C'$.  For such an element, let $Y_1,\ldots,Y_n$ be the subsurfaces in the factor $\bY$ provided by Lemma~\ref{L:linear factor} with $d_{Y_i}(\mu,g\mu) \geq \beta_0 = \beta + 5b$ so that
\begin{equation} \label{E:basic factor bound}
|g| \leq K' \sum_{i=1}^n d_{Y_i}(\mu,g\mu).
\end{equation}
Let $id=g_0' \prec g_1' \prec \ldots \prec g_n' = g$ be the prefixes of a geodesic for $g$ so that Lemma~\ref{L:path traversed in order} holds.  We also let $0 = i_0 < i_1 < i_2 < \ldots < i_k = n$ and $Z_1',\ldots,Z_k'$ be the subsurfaces constructed from $Y_1,\ldots,Y_n$ as above, so that $Y_{i_{j-1} + 1},\ldots,Y_{i_j}\subset Z_j'$, for all $1 \leq j \leq k$.

Since $\beta > 5b$, and $d_{Y_i}(\mu,g\mu) \geq \beta_0 = \beta + 5b$, we have $d_{Y_i}(\mu,g\mu) - 5b \geq \tfrac12 d_{Y_i}(\mu,g\mu)$, for all $i$.  Combining this with Lemma~\ref{L:path traversed in order} we conclude
\begin{equation} \label{E:pre index-change sum} \sum_{i = i_{j-1}+1}^{i_j} \!\!\! d_{Y_i}(g_{i_{j-1}}'\mu,g_{i_j}'\mu) \geq \!\! \sum_{i = i_{j-1}+1}^{i_j} \!\!\! \left( d_{Y_i}(\mu,g\mu)-5b\right)  \geq \tfrac12 \!\!\! \sum_{i = i_{j-1}+1}^{i_j} \!\!\! d_{Y_i}(\mu,g\mu).
\end{equation}

Set $g_0 = id = g_0'$, and for any $1 \leq j \leq k$, define $g_j = (g_{i_{j-1}}')^{-1}g_{i_j}'$ and $Z_j = (g_0 \cdots g_{j-1})^{-1}(Z_j')$.  By Lemma~\ref{L:containers bound main distance}, (i) holds.
Furthermore, by induction, $g_{i_j}' = g_1 \cdots g_j$, $g= g_1g_2\cdots g_k$, and $|g| = |g_1| + \cdots + |g_k|$.  Therefore, part (ii) follows.

Observe that for $i_{j-1}+1 \leq i \leq i_j$, we have $(g_0 \cdots g_{j-1})^{-1}(Y_i) \subset Z_j$ and
\[ d_{(g_0 \cdots g_{j-1})^{-1}Y_i}(\mu,g_j\mu) = d_{Y_i}(g_0\cdots g_{j-1}\mu,g_0 \cdots g_j\mu) = d_{Y_i}(g_{i_{j-1}}'\mu,g_{i_j}'\mu) \geq \beta.\]
Combining this, (\ref{E:pre index-change sum}), and Theorem~\ref{T:distance formula} (Distance Formula), we have
\begin{eqnarray*} 
\kappa d_{\mathcal M(Z_j)}(\mu,g_j\mu) & \geq & \sum_{Y \subseteq Z_j} \{\!\{ d_Y(\mu,g_j\mu) \}\!\}_\beta \\
& \geq & \sum_{i = i_{j-1}+1}^{i_j} d_{(g_0 \cdots g_{j-1})^{-1}Y_i}(\mu,g_j\mu) \\
& \geq & \sum_{i = i_{j-1}+1}^{i_j} d_{Y_i}(g_1\cdots g_{j-1}\mu,g_1 \cdots g_j\mu) \\
& \geq & \tfrac12 \sum_{i = i_{j-1}+1}^{i_j} d_{Y_i}(\mu,g\mu).
\end{eqnarray*}
Recall that $K = 2\kappa K'$.
By summing this inequality over $j$, (\ref{E:basic factor bound}) implies
\[ K \sum_{j=1}^k d_{\mathcal M(Z_j)}(\mu,g_j\mu) 
= 2 \kappa K' \sum_{j=1}^k d_{\mathcal M(Z_j)}(\mu,g_j\mu) 
\geq K' \sum_{i=1}^n d_{Y_i}(\mu,g\mu) \geq |g|. \]
This proves (iii), and completes the proof of the proposition.
\end{proof}

\section{Proof of the Main Theorem}

We are now ready for the proof of the
\begin{theorem:main} A subgroup $G< \Mod(S)$ is convex cocompact if and only if it is finitely generated, undistorted, and purely pseudo-Anosov.
\end{theorem:main}
\begin{proof} If $G$ is convex cocompact, then, by Proposition~\ref{P:ccc is qi in C}, $G$ is finitely generated and any orbit map $G \to \C(S)$ is a quasi-isometric embedding.  Combining this with Proposition~\ref{P:Lipschitz projections} (Lipschitz projection), we see that the orbit map to $\M(S)$ is a quasi-isometric embedding.  By Proposition~\ref{P:MCGMG}, $G$ is undistorted.

Now suppose that $G < \Mod(S)$ is a finitely generated, undistorted, purely pseudo-Anosov subgroup and let $K,C$ be as in Proposition~\ref{P:linear marking subsurface} (Linearly summing projections).   Without loss of generality, we may assume $G$ is torsion free.
Choose $c < \frac{1}{2K}$ and let $R > 1$ be as in Proposition~\ref{P:linear projection reducible} (Linear projection reducibility). 

If, for all $g \in G - \{id\}$, $\frac{d_S(\mu,g\mu)}{|g|}$ is uniformly bounded below, then $G$ is convex cocompact, and we are done.  Therefore, we assume that this quotient can be made arbitrarily small, and derive a contradiction.  Specifically, we assume that there exists $g \in G$ with $|g| \geq \max\{R,C\}$ such that
\[ \frac{d_S(\mu,g\mu)}{|g|} < \frac{1}{2K^2R} \leq \frac1K.\]

Observe that $|g| > K d_S(\mu,g\mu)$, so that the second conclusion of Proposition~\ref{P:linear marking subsurface} (Linearly summing projections) holds.   Let $Z_1,\ldots,Z_k$ and $g_1,\ldots,g_k$ be as in that proposition and set
\[ J_s = \{ j \mid d_{\mathcal M(Z_j)}(\mu,g_j\mu) < c |g_j| \}, \]
and
\[ J_\ell = \{j \mid d_{\mathcal M(Z_j)}(\mu,g_j\mu) \geq c|g_j|\}.\]
If for any $j \in J_\ell$, $|g_j| \geq R$, then by
Proposition~\ref{P:linear projection reducible} (Linear projection reducibility), $G$ contains a
reducible element, a contradiction.  Therefore, we may assume that $|g_j| < R$ for all $j \in J_\ell$.  By Proposition~\ref{P:linear marking subsurface} (Linearly summing projections), we have
\begin{eqnarray*} |g| & \leq & K \sum_{j=1}^k d_{\mathcal M(Z_j)}(\mu,g_j\mu) \\
& = & K \left( \sum_{j \in J_s} d_{\mathcal M(Z_j)}(\mu,g_j\mu) + \sum_{j \in J_\ell} d_{\mathcal M(Z_j)}(\mu,g_j\mu) \right) \\
& < & K \left( \sum_{j \in J_s} c |g_j| + \sum_{j \in J_\ell} K R \right) \\
& \leq & K ( c|g| + k K R) \\
& \leq & K ( c|g| + d_S(\mu,g\mu) K R)
\end{eqnarray*}
Dividing both sides by $|g|$, we find
\[ 1 \leq cK + \frac{d_S(\mu,g\mu) K^2R}{|g|} < \frac{K}{2K} +  \frac{K^2R}{2K^2R} = \frac12 + \frac12 = 1.\]
This is a contradiction, which completes the proof.
\end{proof}

  \bibliographystyle{alpha}
  \bibliography{main}

\begin{thebibliography}{BKMM12}

\bibitem[BBF15]{BBF}
Mladen Bestvina, Ken Bromberg, and Koji Fujiwara.
\newblock Constructing group actions on quasi-trees and applications to mapping
  class groups.
\newblock {\em Publ. Math. Inst. Hautes \'Etudes Sci.}, 122:1--64, 2015.

\bibitem[Beh06]{Behrstock}
Jason~A. Behrstock.
\newblock Asymptotic geometry of the mapping class group and {T}eichm\"uller
  space.
\newblock {\em Geom. Topol.}, 10:1523--1578, 2006.

\bibitem[Bes]{bestvina:problems}
M.~Bestvina.
\newblock Questions in geometric group theory, {U}pdated {J}uly 2004.
\newblock \texttt{https://www.math.utah.edu/$\sim$bestvina/}.

\bibitem[BKMM12]{BKMM}
Jason Behrstock, Bruce Kleiner, Yair Minsky, and Lee Mosher.
\newblock Geometry and rigidity of mapping class groups.
\newblock {\em Geom. Topol.}, 16(2):781--888, 2012.

\bibitem[CLM12]{CLM:RAAGS}
Matt~T. Clay, Christopher~J. Leininger, and Johanna Mangahas.
\newblock The geometry of right-angled {A}rtin subgroups of mapping class
  groups.
\newblock {\em Groups Geom. Dyn.}, 6(2):249--278, 2012.

\bibitem[DKL14]{DKL:3-mfd}
Spencer Dowdall, Autumn~E. Kent, and Christopher~J. Leininger.
\newblock Pseudo-{A}nosov subgroups of fibered 3-manifold groups.
\newblock {\em Groups Geom. Dyn.}, 8(4):1247--1282, 2014.

\bibitem[DT15]{DT:stable}
Matthew~Gentry Durham and Samuel~J. Taylor.
\newblock Convex cocompactness and stability in mapping class groups.
\newblock {\em Algebr. Geom. Topol.}, 15(5):2839--2859, 2015.

\bibitem[FM02]{FM:CC}
Benson Farb and Lee Mosher.
\newblock Convex cocompact subgroups of mapping class groups.
\newblock {\em Geom. Topol.}, 6:91--152 (electronic), 2002.

\bibitem[FM10]{farb:MCG}
B.~Farb and D.~Margalit.
\newblock {\em A primer on mapping class groups}.
\newblock Princeton Univ. Press, Princeton, N.J., 2010.

\bibitem[Ham]{Hamenstadt}
Ursula Hamenst{\"a}dt.
\newblock {Word hyperbolic extensions of surface groups}.
\newblock Preprint, \texttt{arXiv:math.GT/0505244}.

\bibitem[KL07]{KL:survey}
Autumn~E. Kent and Christopher~J. Leininger.
\newblock Subgroups of mapping class groups from the geometrical viewpoint.
\newblock In {\em In the tradition of {A}hlfors-{B}ers. {IV}}, volume 432 of
  {\em Contemp. Math.}, pages 119--141. Amer. Math. Soc., Providence, RI, 2007.

\bibitem[KL08]{KL:CC}
Autumn~E. Kent and Christopher~J. Leininger.
\newblock Shadows of mapping class groups: capturing convex cocompactness.
\newblock {\em Geom. Funct. Anal.}, 18(4):1270--1325, 2008.

\bibitem[KLS09]{KLS:trees}
Autumn~E. Kent, Christopher~J. Leininger, and Saul Schleimer.
\newblock Trees and mapping class groups.
\newblock {\em J. Reine Angew. Math.}, 637:1--21, 2009.

\bibitem[KMT17]{KMT:RAAG}
Thomas Koberda, Johanna Mangahas, and Samuel~J. Taylor.
\newblock The geometry of purely loxodromic subgroups of right-angled {A}rtin
  groups.
\newblock {\em Trans. Amer. Math. Soc.}, 369(11):8179--8208, 2017.

\bibitem[Man10]{Mangahas}
Johanna Mangahas.
\newblock Uniform uniform exponential growth of subgroups of the mapping class
  group.
\newblock {\em Geom. Funct. Anal.}, 19(5):1468--1480, 2010.

\bibitem[MM99]{MM:CC1}
H.A. Masur and Y.N. Minsky.
\newblock Geometry of the complex of curves. {I}. {H}yperbolicity.
\newblock {\em Invent. Math.}, 138(1):103--149, 1999.

\bibitem[MM00]{MM:CC2}
H.A. Masur and Y.~N. Minsky.
\newblock Geometry of the complex of curves. {II}. {H}ierarchical structure.
\newblock {\em Geom. Funct. Anal.}, 10(4):902--974, 2000.

\bibitem[MS12]{MjSardar}
Mahan Mj and Pranab Sardar.
\newblock A combination theorem for metric bundles.
\newblock {\em Geom. Funct. Anal.}, 22(6):1636--1707, 2012.

\bibitem[MT16]{MT:RAAG}
Johanna Mangahas and Samuel~J. Taylor.
\newblock Convex cocompactness in mapping class groups via quasiconvexity in
  right-angled artin groups.
\newblock {\em Proc. London Math. Soc.}, 112(5):855--881, 2016.

\end{thebibliography}

\end{document}